\documentclass[10pt]{article}
\usepackage{graphicx}
\usepackage{amsmath}
\usepackage{amsfonts}
\usepackage{amsthm}
\usepackage[T1]{fontenc}
\usepackage{url}
\usepackage{color}
\usepackage[margin=1in]{geometry}
\usepackage{amssymb,bm}
\usepackage{amsmath}
\usepackage{amsthm}
\usepackage{graphicx}
\usepackage[active]{srcltx} 
\usepackage{hyperref}
\hypersetup{pdfborder=0 0 0}

\newtheorem{theorem}{{\sc Theorem}}[section]

\newtheorem{lemma}[theorem]{{\sc Lemma}}

\def\XXint#1#2#3{{\setbox0=\hbox{$#1{#2#3}{\int}$ }
\vcenter{\hbox{$#2#3$ }}\kern-.6\wd0}}

\newcommand{\Gk}{\kappa}

\newcommand{\Gth}{\theta}

\bmdefine\BGa{\alpha}
\bmdefine\BGb{\beta}
\bmdefine\BGd{\delta}
\bmdefine\BGe{\epsilon}
\bmdefine\BGve{\varepsilon}
\bmdefine\BGf{\phi}
\bmdefine\BGvf{\varphi}
\bmdefine\BGg{\gamma}
\bmdefine\BGc{\chi}
\bmdefine\BGi{\iota}
\bmdefine\BGk{\kappa}
\bmdefine\BGl{\lambda}
\bmdefine\BGn{\eta}
\bmdefine\BGm{\mu}
\bmdefine\BGv{\nu}
\bmdefine\BGp{\pi}
\bmdefine\BGth{\theta}
\bmdefine\BGvth{\vartheta}
\bmdefine\BGr{\rho}
\bmdefine\BGvr{\varrho}
\bmdefine\BGs{\sigma}
\bmdefine\BGvs{\varsigma}
\bmdefine\BGt{\tau}
\bmdefine\BGj{\tau}
\bmdefine\BGu{\upsilon}
\bmdefine\BGo{\omega}
\bmdefine\BGx{\xi}
\bmdefine\BGy{\psi}
\bmdefine\BGz{\zeta}
\bmdefine\BGD{\Delta}
\bmdefine\BGF{\Phi}
\bmdefine\BGG{\Gamma}
\bmdefine\BGL{\Lambda}
\bmdefine\BGP{\Pi}
\bmdefine\BGT{\Theta}
\bmdefine\BGS{\Sigma}
\bmdefine\BGU{\Upsilon}
\bmdefine\BGO{\Omega}
\bmdefine\BGX{\Xi}
\bmdefine\BGY{\Psi}

\bmdefine\BCA{{\mathcal A}}
\bmdefine\BCB{{\mathcal B}}
\bmdefine\BCC{{\mathcal C}}
\bmdefine\BCD{{\mathcal D}}
\bmdefine\BCE{{\mathcal E}}
\bmdefine\BCF{{\mathcal F}}
\bmdefine\BCG{{\mathcal G}}
\bmdefine\BCH{{\mathcal H}}
\bmdefine\BCI{{\mathcal I}}
\bmdefine\BCJ{{\mathcal J}}
\bmdefine\BCK{{\mathcal K}}
\bmdefine\BCL{{\mathcal L}}
\bmdefine\BCM{{\mathcal M}}
\bmdefine\BCN{{\mathcal N}}
\bmdefine\BCO{{\mathcal O}}
\bmdefine\BCP{{\mathcal P}}
\bmdefine\BCQ{{\mathcal Q}}
\bmdefine\BCR{{\mathcal R}}
\bmdefine\BCS{{\mathcal S}}
\bmdefine\BCT{{\mathcal T}}
\bmdefine\BCU{{\mathcal U}}
\bmdefine\BCV{{\mathcal V}}
\bmdefine\BCW{{\mathcal W}}
\bmdefine\BCX{{\mathcal X}}
\bmdefine\BCY{{\mathcal Y}}
\bmdefine\BCZ{{\mathcal Z}}

\bmdefine\Bzr{ 0}
\bmdefine\Ba{ a}
\bmdefine\Bb{ b}
\bmdefine\Bc{ c}
\bmdefine\Bd{ d}
\bmdefine\Be{ e}
\bmdefine\Bf{ f}
\bmdefine\Bg{ g}
\bmdefine\Bh{ h}
\bmdefine\Bi{ i}
\bmdefine\Bj{ j}
\bmdefine\Bk{ k}
\bmdefine\Bl{ l}
\bmdefine\Bm{ m}
\bmdefine\Bn{ n}
\bmdefine\Bo{ o}
\bmdefine\Bp{ p}
\bmdefine\Bq{ q}
\bmdefine\Br{ r}
\bmdefine\Bs{ s}
\bmdefine\Bt{ t}
\bmdefine\Bu{ u}
\bmdefine\Bv{ v}
\bmdefine\Bw{ w}
\bmdefine\Bx{ x}
\bmdefine\By{ y}
\bmdefine\Bz{ z}
\bmdefine\BA{ A}
\bmdefine\BB{ B}
\bmdefine\BC{ C}
\bmdefine\BD{ D}
\bmdefine\BE{ E}
\bmdefine\BF{ F}
\bmdefine\BG{ G}
\bmdefine\BH{ H}
\bmdefine\BI{ I}
\bmdefine\BJ{ J}
\bmdefine\BK{ K}
\bmdefine\BL{ L}
\bmdefine\BM{ M}
\bmdefine\BN{ N}
\bmdefine\BO{ O}
\bmdefine\BP{ P}
\bmdefine\BQ{ Q}
\bmdefine\BR{ R}
\bmdefine\BS{ S}
\bmdefine\BT{ T}
\bmdefine\BU{ U}
\bmdefine\BV{ V}
\bmdefine\BW{ W}
\bmdefine\BX{ X}
\bmdefine\BY{ Y}
\bmdefine\BZ{ Z}

\begin{document}
\title{The asymptotically sharp Korn interpolation and second inequalities for shells}
\author{D. Harutyunyan}
\maketitle

\begin{abstract}
We consider shells in three dimensional Euclidean space which have bounded principal curvatures. We prove Korn's interpolation (or the so called first and a half\footnote{The inequality first introduced in [\ref{bib:Gra.Har.1}]}) and second inequalities on that kind of shells for $\Bu\in W^{1,2}$ vector fields, imposing no boundary or normalization conditions on $\Bu.$ The constants in the estimates are optimal in terms of the asymptotics in the shell thickness $h,$  having the scalings $h$ or $O(1).$ The Korn interpolation inequality reduces the problem of deriving any linear Korn type estimate for shells to simply proving a Poincar\'e type estimate with the symmetrized gradient on the right hand side. In particular this applies to linear geometric rigidity estimates for shells, i.e., Korn's fist inequality without boundary conditions.
\end{abstract}

\section{Introduction}
\label{sec:1}
A shell of thickness $h$ in three dimensional Euclidean space is given by $\Omega=\{x+t\Bn(x) \ : \ x\in S,\ t\in [-h/2,h/2]\},$ where $S\subset\mathbb R^3$ is a bounded and connected smooth enough regular surface with a unit normal $\Bn(x)$ at the point $x\in S.$ The surface $S$ is called the mid-surface of the shell $\Omega.$ Understanding the rigidity of a shell is one of the challenges in nonlinear elasticity, where there are still many open questions. Unlike the situation for shells in general, the rigidity of plates has been quite well understood by Friesecke, James and M\"uller in their celebrated papers [\ref{bib:Fri.Jam.Mue.1},\ref{bib:Fri.Jam.Mue.2}]. It is known that the rigidity of a shell $\Omega$ is closely related to the optimal Korn's constant in the nonlinear (in some cases linear) first Korn's inequality [\ref{bib:Fri.Jam.Mue.2},\ref{bib:Gra.Har.1}], which is a geometric rigidity estimate for $\Bu\in H^1(\Omega)$ fields 
[\ref{bib:Kohn},\ref{bib:Fri.Jam.Mue.1},\ref{bib:Ciarlet1},\ref{bib:Cia.Mar},\ref{bib:Cia.Mar.Mar.1}]. Depending on the problem, the field $\Bu\in H^1$ may or may not satisfy boundary conditions, e.g. [\ref{bib:Kohn},\ref{bib:Fri.Jam.Mue.2},\ref{bib:Gra.Har.1}]. Finding the optimal constants in Korn's inequalities is a central task in problems concerning shells in general. The Friesecke-James-M\"uller estimate reads as follows:  \textit{Assume $\Omega\subset\mathbb R^3$ is open bounded connected and Lipschitz. Then there exists a constant $C_I=C_I(\Omega),$ such that for every vector field $\Bu\in H^1(\Omega),$ there exists a constant rotation $\BR\in SO(3)$, such that
\begin{equation}
\label{1.1}
\|\nabla\Bu-\BR\|^2\leq C_{I}\int_\Omega\mathrm{dist}^2(\nabla\Bu(x),SO(3))dx.
\end{equation}
}
The linearization of (\ref{1.1}) around the identity matrix is Korn's first inequality [\ref{bib:Korn.1},\ref{bib:Korn.2},\ref{bib:Kon.Ole.2},\ref{bib:Fri.Jam.Mue.1},\ref{bib:Ciarlet1}] without boundary conditions and reads as follows: \textit{Assume $\Omega\subset\mathbb R^n$ is open bounded connected and Lipschitz. Then there exists a constant $C_{II}=C_{II}(\Omega),$ depending only on $\Omega,$ such that for every vector field $\Bu\in H^1(\Omega)$ there exists a skew-symmetric matrix $\BA\in \mathbb R^{n\times n,}$ i.e., $A+A^T=0,$ such that
\begin{equation}
\label{1.2}
\|\nabla\Bu-\BA\|_{L^2(\Omega)}^2\leq C_{II}\|e(\Bu)\|_{L^2(\Omega)}^2,
\end{equation}
where $e(\Bu)=\frac{1}{2}(\nabla\Bu+\nabla\Bu^T)$ is the symmetrized gradient (the strain in linear elasticity).} The estimate (\ref{1.2}) is traditionally proven by using Korn's second inequality, that reads as follows: \textit{Assume $\Omega\subset\mathbb R^n$ is open bounded connected and Lipschitz. Then there exists a constant $C=C(\Omega),$ depending only on $\Omega,$ such that for every vector field $\Bu\in H^1(\Omega)$ there holds:}
\begin{equation}
\label{1.3}
\|\nabla\Bu\|_{L^2(\Omega)}^2\leq C(\|\Bu\|_{L^2(\Omega)}^2+\|e(\Bu)\|_{L^2(\Omega)}^2).
\end{equation}
It is known that if $\Omega$ is a thin domain with thickness $h,$ then in general the optimal constants $C$ in all inequalities (\ref{1.1})-(\ref{1.3}) blow up as $h\to 0.$ In particular, if $\Omega$ is a plate given by $\Omega=\omega\times(0,h),$ where $\omega\subset\mathbb R^2$ is open bounded connected and Lipschitz, then as proven in [\ref{bib:Fri.Jam.Mue.2}] one has $C_I=c_1(\omega)h^2$ and $C_{II}=c_2(\omega)h^2$ asymptotically as $h\to 0.$ While the asymptotics of $C_{II}$ is known in the case when $\Bu$ satisfies zero Dirichlet boundary conditions on the thin face of the shell [\ref{bib:Gra.Har.2},\ref{bib:Harutyunyan.2}] ($C_{II}$ scaling like $h^2,$ $h^{3/2},$ $h^{4/3}$ or $h^{1}$), it is open for general fields $\Bu\in H^1(\Omega).$ In this work we are concerned with the asymptotics of the constant $C$ in (\ref{1.3}) or more precisely in the so called Korn interpolation inequality, or the first-and-a-half Korn inequality [\ref{bib:Gra.Har.1}], in the general case when $\Omega$ is a shell. The statements solving the problem practically completely appear in the next section.

\section{Main Results}
\setcounter{equation}{0}
\label{sec:2}

We first introduce the main notation and definitions. We will assume throughout this work that the mid-surface $S$ of the shell $\Omega$ is connected, compact, regular and of class $C^3$ up to its boundary. We also assume that $S$ has a finite atlas of patches $S\subset\cup_{i=1}^k\Sigma_i$ such that each  patch $\Sigma_i$ can be parametrized by the principal variables $z$ and $\Gth$ ($z=$constant and $\Gth=$constant are the principal lines on $\Sigma_i$) that change in the ranges $z\in [z_i^1(\Gth),z_i^2(\Gth)]$ for $\Gth\in [0,\omega_i],$ where $\omega_i>0$ for $i=1,2,\dots,k.$ Moreover, the functions $z_i^1(\Gth)$ and $z_i^2(\Gth)$ satisfy the conditions
\begin{align}
\label{2.1}
&\min_{1\leq i\leq k}\inf_{\Gth\in [0,\omega_i]}[z_i^2(\Gth)-z_i^1(\Gth)]=l>0,\quad\max_{1\leq i\leq k}\sup_{\Gth\in [0,\omega_i]}[z_i^2(\Gth)-z_i^1(\Gth)]=L<\infty,\\ \nonumber
&\max_{1\leq i\leq k}\left(\|z_i^1\|_{W^{1,\infty}[0,\omega_i]}+\|z_i^2\|_{W^{1,\infty}[0,\omega_i]}\right)=Z<\infty.
\end{align}
Since there will be no condition imposed on the vector field $\Bu\in H^1(\Omega),$ (see Theorem~\ref{th:2.1}), we can restrict ourselves to a single patch $\Sigma\subset S$ and denote it by $S$ for simplicity. If the parametrization of $S$ is $\Br=\Br(\Gth,z)$ and $\Bn$ is the unit normal to $S,$ denoting the normal variable by $t$ and $A_{z}=\left|\frac{\partial \Br}{\partial z}\right|, A_{\Gth}=\left|\frac{\partial \Br}{\partial\Gth}\right|$ we get
\begin{equation}
\label{2.2}
\nabla\Bu=
\begin{bmatrix}
  u_{t,t} & \dfrac{u_{t,\Gth}-A_{\Gth}\Gk_{\Gth}u_{\Gth}}{A_{\Gth}(1+t\Gk_{\Gth})} &
\dfrac{u_{t,z}-A_{z}\Gk_{z}u_{z}}{A_{z}(1+t\Gk_{z})}\\[3ex]
u_{\Gth,t}  &
\dfrac{A_{z}u_{\Gth,\Gth}+A_{z}A_{\Gth}\Gk_{\Gth}u_{t}+A_{\Gth,z}u_{z}}{A_{z}A_{\Gth}(1+t\Gk_{\Gth})} &
\dfrac{A_{\Gth}u_{\Gth,z}-A_{z,\Gth}u_{z}}{A_{z}A_{\Gth}(1+t\Gk_{z})}\\[3ex]
u_{ z,t}  & \dfrac{A_{z}u_{z,\Gth}-A_{\Gth,z}u_{\Gth}}{A_{z}A_{\Gth}(1+t\Gk_{\Gth})} &
\dfrac{A_{\Gth}u_{z,z}+A_{z}A_{\Gth}\Gk_{z}u_{t}+A_{z,\Gth}u_{\Gth}}{A_{z}A_{\Gth}(1+t\Gk_{z})}
\end{bmatrix}
\end{equation}
in the orthonormal local basis $(\Bn,\Be_\Gth,\Be_z),$ where $\Gk_{z}$ and $\Gk_{\Gth}$ are the two principal curvatures. Here we use the notation $f_{,\alpha}$ for the partial derivative $\frac{\partial}{\partial\alpha}$ inside the gradient matrix of a vector field $\Bu\colon\Omega\to\mathbb R^3.$ The gradient on $S$ or the so called simplified gradient denoted by $\BF$ is obtained from (\ref{2.2}) by putting $t=0.$ We will work with $\BF$ and then pass to $\nabla\Bu$ using their closeness to the order of $h$ due to the smallness of the variable $t.$ In this paper all norms $\|\cdot\|$ are $L^{2}$ norms and the $L^2$ inner product of two functions $f,g\colon\Omega\to\mathbb R$ will be given by $(f,g)_{\Omega}=\int_{\Omega}A_zA_\Gth f(t,\Gth,z)g(t,\Gth,z)d\Gth dzdt,$ which gives rise to the norm $\|f\|_{L^2(\Omega)}$.
In what follows in the below theorems, the constants $h_0>0$ and $C>0$ will depend only on the shell mid-surface parameters, which are the quantities $\omega,l,L,Z,a=\min_{D}(A_\Gth,A_z), A=\|A_\Gth\|_{W^{2,\infty}(D)}+\|A_z\|_{W^{2,\infty}(D)}$ and  $k=\|\Gk_\Gth\|_{W^{1,\infty}(D)}+\|\Gk_z\|_{W^{1,\infty}(D)},$ where $D=\{(\Gth,z)\ : \ \Gth\in [0,\omega], z\in[z^1(\Gth),z^2(\Gth)]\}.$ Our results are Korn's interpolation and second inequalities for the shell $\Omega,$ providing sharp Ansatz-free lower bounds for displacements $\Bu\in H^1(\Omega,\mathbb R^3)$ imposing no boundary condition on the field $\Bu.$ The estimates are also proven to be asymptotically optimal as $h\to 0.$
\begin{theorem}[Korn's interpolation inequality]
\label{th:2.1}
There exists constants $h_0,C>0,$ such that Korn's interpolation inequality holds:
\begin{equation}
  \label{2.3}
\|\nabla\Bu\|^2\leq C\left(\frac{\|\Bu\cdot\Bn\|\cdot\|e(\Bu)\|}{h}+\|\Bu\|^2+\|e(\Bu)\|^2\right),
\end{equation}
for all $h\in(0,h_0)$ and $\Bu=(u_t,u_\Gth,u_z)\in H^1(\Omega),$ where $\Bn$ is the unit normal to the mid-surface $S.$
Moreover, the exponent of $h$ in the inequality (\ref{2.3}) is optimal for any shell $\Omega$ satisfying the above imposed regularity condition together with (\ref{2.1}), i.e., there exists a displacement $\Bu\in H^1(\Omega,\mathbb R^3)$ realizing the asymptotics of $h$ in (\ref{2.3}).
\end{theorem}
\begin{theorem}[Korn's second inequality]
\label{th:2.2}
We get by the Cauchy-Schwartz inequality from (\ref{2.3}) the following Korn's second inequality for shells: There exists constants $h_0,C>0,$ such that Korn's second inequality holds:
\begin{equation}
  \label{2.4}
\|\nabla\Bu\|^2\leq \frac{C}{h}(\|\Bu\|^2+\|e(\Bu)\|^2),
\end{equation}
for all $h\in(0,h_0)$ and $\Bu=(u_t,u_\Gth,u_z)\in H^1(\Omega).$ Moreover, the exponent of $h$ in the inequality (\ref{2.4}) is optimal for any shell $\Omega$ satisfying the above imposed regularity condition together with (\ref{2.1}), i.e., there exists a displacement $\Bu\in H^1(\Omega,\mathbb R^3)$ realizing the asymptotics of $h$ in (\ref{2.4}).
\end{theorem}

\section{The key lemma}
\label{sec:3}
\setcounter{equation}{0}
In this section we prove a gradient separation estimate for harmonic functions in two dimensional thin rectangles, which is one of the key estimates in the proof of Theorem~\ref{th:2.1}.
\begin{lemma}
\label{lem:3.1}
Assume $h,b>0$ such that $b>3h.$ Denote $R_b=(0,h)\times(0,b)\subset\mathbb R^2.$ There exists a universal constat $C>0,$ such that any harmonic function $w\in C^2(R_b)$ fulfills the inequality
\begin{equation}
\label{3.1}
\|w_y\|_{L^2(R_b)}^2\leq C\left(\frac{1}{h}\|w\|_{L^2(R_b)}\cdot\|w_x\|_{L^2(R_b)}+\frac{1}{b^2}\|w\|_{L^2(R_b)}^2+\|w_x\|_{L^2(R_b)}^2\right).
\end{equation}
\end{lemma}
\begin{proof}[Sketch of proof]
We divide the proof into four steps for the convenience of the reader. Let us point out that all the norms in the proof are $L^2(R_b)$ unless specified otherwise.\\
\textbf{Step 1. An estimate on rectangles.} \textit{Assume $h>0$ and denote $R=(0,h)\times(0,1)\subset\mathbb R^2.$ There exists a universal constat $c>0$ such that any harmonic function $w\in C^2(R)$ fulfills the inequality
\begin{equation}
\label{3.2}
\|w_y-a\|_{L^2(R)}\leq\frac{c}{h}\|w_x\|_{L^2(R)},
\end{equation}
where $a=\frac{1}{|R|}\int_R w_y$ is the average of $w_y$ over the rectangle $R.$
}
Estimate (\ref{3.2}) is derived from the linear version of (\ref{1.1}) for plates, i.e., the estimate (\ref{1.2}) for $\Omega=\omega\times(0,h)$ as mentioned in the previous section. Indeed, considering the plate $\Omega=R\times (0,1)\subset\mathbb R^3,$ and the displacement
$u_1(x,y)=w(x,y),\ u_2(x,y)=-\int_0^x w_y(t,y)dt+\int_0^y w_x(0,z)dz,\ u_3\equiv 0,$ one gets (\ref{3.2}) with $a_{12}$ instead of $a,$ but the quantity $\|w_y-\lambda\|_{L^2(R)}^2$ is minimized at $\lambda=a.$ Therefore (\ref{3.2}) follows.\\
\textbf{Step 2. An interior estimate on $w_y.$} \textit{There exists an absolute constant $C>0$ such that for any harmonic function $w\in C^2(R_b)$ the inequality holds:}
\begin{equation}
\label{3.3}
\int_{(h/4,3h/4)\times(0,b)}|w_y|^2\leq C\left(\frac{1}{h}\|w\|\cdot\|w_x\|+\frac{1}{b^2}\|w\|^2+\|w_x\|^2\right).
\end{equation}
Let $z\in(h,b/2)$ be a parameter and let $\varphi(y)\colon[0,b]\to [0,1]$ be a smooth cutoff function such that
$\varphi(y)=1$ for $y\in[z,b-z]$ and $|\nabla \varphi(y)|\leq \frac{2}{z}$ for $y\in[0,b].$ Next for $t\in (0,h/2)$ we denote $R_{t,z}=(h/2-t,h/2+t)\times(z,b-z),$ $R_{z}^{top}=(0,h)\times(b-z,b)$ and $R_{z}^{bot}=(0,h)\times(0,z).$ We multiply the equality $-\Delta w=0$ in $R_{b}$ by $\varphi w$ and integrate the obtained identity first by parts over $R_{t,b}$ and then in $t$ over $(h/4,h/2)$ to get the estimate
\begin{equation}
\label{3.4}
\int_{R_{h/4,z}}|\nabla w|^2\leq \frac{4}{h}\int_{R_b}|ww_x|+\frac{1}{\epsilon^2z^2}\int_{R_{z}^{bot}\cup R_{z}^{top}}w^2
+\epsilon^2\int_{R_{z}^{bot}\cup R_{z}^{top}}w_y^2,
\end{equation}
where $\epsilon>0$ is a parameter yet to be chosen. By the invariance of (\ref{3.2}) under the variable change $(x,y)\to (\lambda x,\lambda y),$ we have for some $a_1,a_2\in\mathbb R,$
\begin{equation}
\label{3.5}
\int_{R_{2z}^{bot}} |w_y-a_1|^2\leq \frac{cz^2}{h^2}\int_{R_{2z}^{bot}} |w_x|^2,\quad\text{and}\quad
\int_{R_{2z}^{top}} |w_y-a_2|^2\leq \frac{cz^2}{h^2}\int_{R_{2z}^{top}} |w_x|^2,
\end{equation}
which gives together with the triangle inequality the estimates
\begin{equation}
\label{3.6}
\int_{R_{h/4,z}}|\nabla w|^2\geq \frac{hz}{4}(a_1^2+a_2^2)-\frac{cz^2}{h^2}\int_{R_{2z}^{bot}} |w_x|^2-\frac{cz^2}{h^2}\int_{R_{2z}^{top}}|w_x|^2.
\end{equation}
An application of the triangle inequality to $\int_{R_{z}^{bot}}w_y^2, \int_{R_{z}^{top}}w_y^2$ in (\ref{3.4}) and utilization of (\ref{3.5}) and (\ref{3.6}) derives from (\ref{3.4}) for the value $\epsilon=1/4$ the estimate
\begin{equation}
\label{3.7}
\frac{hz}{8}(a_1^2+a_2^2)\leq \frac{4}{h}\int_{R_b}|ww_x|+\frac{16}{z^2}\int_{R_{z}^{bot}\cup R_{z}^{top}}w^2
+\frac{2cz^2}{h^2}\int_{R_{2z}^{bot}\cup R_{2z}^{top}}|w_x|^2.
\end{equation}
Newt we combine (\ref{3.4}) (for $\epsilon=1$), (\ref{3.5}) and (\ref{3.7}) to get the key interior estimate
\begin{equation}
\label{3.8}
\int_{R_{h/4,0}}|w_y|^2\leq C\left(\frac{1}{h}\int_{R_b}|ww_x|+\frac{1}{z^2}\|w\|^2
+\frac{z^2}{h^2}\|w_x\|^2\right).
\end{equation}
It remains to minimize the right hand side of (\ref{3.8}) subject to the constraint $h\leq z<b/2$ on the parameter $z$to get (\ref{3.3}) The procedure is standard and is left to the reader.\\
\textbf{Step 3. An estimate near the horizontal boundary of $R_b$.} \textit{There exists an absolute constant $C>0,$ such that for any harmonic function $w\in C^2(R)$ the inequality holds:}
\begin{equation}
\label{3.9}
\int_{R_{h}^{bot}\cup R_{h}^{top}}|w_y|^2\leq C\left(\frac{1}{h}\int_{R_b}|ww_x|+\frac{1}{b^2}\|w\|^2+\|w_x\|^2\right).
\end{equation}
The proof is similar to Step1 by the utilization of (\ref{3.5}) and (\ref{3.7}).\\
\textbf{Step 4. Proof of (\ref{3.1}).} We recall the following two auxiliary lemmas proven by Kondratiev and Oleinik [\ref{bib:Kon.Ole.2}], see also [\ref{bib:Harutyunyan.1}].
\begin{lemma}
\label{lem:3.2}
Assume $0<a$ and $f\colon[0,2a]\to\mathbb R$ is absolutely continuous. Then the inequality holds:
\begin{equation}
\label{3.10}
\int_0^af^2(t)dt\leq 4\int_a^{2a}f^2(t)dt+4\int_0^{2a}t^2t'^2(t)dt.
\end{equation}
\end{lemma}
\begin{lemma}
\label{lem:3.3}
Let $n\in\mathbb R^n,$ and let $\Omega\subset\mathbb R^n$ be open bounded connected and Lipschitz. Denote $\delta(x)=\mathrm{dist}(x,\partial\Omega).$ Assume $u\in C^2(\Omega)$ is harmonic. Then there holds:
\begin{equation}
\label{3.11}
\|\delta\nabla u\|_{L^2(\Omega)}\leq 2\|\nabla u\|_{L^2(\Omega)}.
\end{equation}
\end{lemma}
Fixing a point $y\in (h,b-h)$ and applying Lemma~\ref{lem:3.2} to the function $w_y(x,y)$ on the segment $[0,h/2]$ as a function in $x,$ we get
\begin{equation}
\label{3.12}
\int_{(0,h/4)\times(h,b-h)}|w_y|^2\leq \int_{(h/4,h/2)\times(h,b-h)}|w_y|^2+4\int_{(0,h/2)\times(h,b-h)}|xw_{xy}|^2.
\end{equation}
Lemma~\ref{lem:3.3} applied to the harmonic function $w_x$ reduces (\ref{3.12}) to the key estimate
\begin{equation}
\label{3.13}
\int_{(0,h/4)\times(h,b-h)}|w_y|^2\leq \int_{(h/4,h/2)\times(h,b-h)}|w_y|^2+16\int_{R_b}|w_{x}|^2.
\end{equation}
It remains to combine a similar estimate for the right part of the rectangle with (\ref{3.13}), (\ref{3.9}) and (\ref{3.3}).
\end{proof}

\section{Proof of the main results}
\label{sec:4}
\setcounter{equation}{0}

\begin{proof}[Sketch of proof of Theorem~\ref{th:2.1}]
Let us point out that throughout this section the constants $h_0,C>0$ will depend only on the quantities $a,A,\omega,l,L,k$ and $Z$ unless specified otherwise. We first prove the estimate with $\BF$ and $e(\BF)$ in place of $\nabla\Bu$ and $e(\Bu)$ in (\ref{2.3}), which we do 
block by block by freezing each of the variables $t,$ $\Gth$ and $z$.\\
\textbf{The block $23$.} We aim to prove the estimate
\begin{equation}
\label{4.1}
\|F_{23}\|^2+\|F_{32}\|^2\leq C(\|\Bu\|^2+\|e(\BF)\|^2).
\end{equation}
Denote $R_t=\{(\Gth,z)\ : \ \Gth\in(0,\omega), z\in(z^1(\Gth),z^2(\Gth))\}$ and assume $\varphi=\varphi(\Gth,z)\in C^1(R_t,\mathbb R)$ satisfies the conditions $0<c_1\leq \varphi(\Gth,z)\leq c_2,\ \|\nabla \varphi(\Gth,z)\|\leq c_3$ for all $(\Gth,z)\in R_t.$ Then, for any displacement $\BU=(u,v)\in H^1(R_t,\mathbb R^2),$ considering the auxiliary vector field $\BW=\left(u,\frac{1}{\varphi}v\right)\colon R_t\to\mathbb R^2,$ one can get from Korn's second inequality [\ref{bib:Kon.Ole.2}], that there exists a constant $c>0,$ depending only on the constants $\omega,l,L,Z$ and $c_i,\ i=1,2,3,$ such that for the matrix 
$
\BM_\varphi=
\begin{bmatrix}
u_x & \varphi u_y\\
v_x & \varphi v_y
\end{bmatrix}
$
fulfills the estimate
\begin{equation}
\label{4.2}
\|\BM_\varphi\|_{L^2(R_t)}^2\leq c(\|e(\BM_\varphi)\|_{L^2(R_t)}^2+\|u\|_{L^2(R_t)}^2+\|v\|_{L^2(R_t)}^2).
\end{equation}
An application of (\ref{4.2}) for $\varphi(\Gth,z)=\frac{A_\Gth}{A_z}$ and $\BU=(u_\Gth,u_z)$ gives (\ref{4.1}).
We combine the estimates for the other two blocks in one by first proving the following Korn-like inequality on thin rectangles, which will be the key estimate for the rest of the proof.
\begin{lemma}
\label{lem:4.1}
For $0<h\leq b/3$ denote $R=(0,h)\times(0,b).$ Given a displacement $\BU=(u(x,y),v(x,y))\in H^1(R,\mathbb R^2),$ the vector fields
$\alpha,\beta\in W^{1,\infty}(R,\mathbb R^2)$ and the function $w\in H^1(R,\mathbb R),$ denote the perturbed gradient as follows:
\begin{equation}
\label{4.3}
\BM=
\begin{bmatrix}
u_{x} & u_{y}+\alpha\cdot\BU\\
v_{x} & v_{y}+\beta\cdot\BU+w
\end{bmatrix}.
\end{equation}
Assume $\epsilon\in (0,1),$ then the following Korn-like interpolation inequality holds:
\begin{equation}
\label{4.4}
\|\BM\|_{L^2(R)}^2\leq C\left(\frac{\|u\|_{L^2(R)}\cdot \|e(\BM)\|_{L^2(R)}}{h}+\|e(\BM)\|_{L^2(R)}^2+\frac{1}{\epsilon}\|\BU\|_{L^2(R)}^2+\epsilon(\|w_{L^2(R)}\|^2+\|w_x\|_{L^2(R)}^2)\right),
\end{equation}
for all $h$ small enough, where $C$ depends only on the quantities $b,$ $\|\alpha\|_{W^{1,\infty}}$ and $\|\beta\|_{W^{1,\infty}}.$
\end{lemma}
\begin{proof}
Let us point out that in the proof of Lemma~\ref{lem:4.1}, the constant $C$ may depend only on $b,$ $\|\alpha\|_{W^{1,\infty}}$ and $\|\beta\|_{W^{1,\infty}}$ as well as the norm $\|\cdot\|$ will be $\|\cdot\|_{L^2(R)}.$ First of all, we can assume by density that $\BU\in C^2(\bar R).$ For functions $f,g\in H^1(R,\mathbb R),$ denote
$
\BM_{f,g}=
\begin{bmatrix}
u_{x} & u_{y}+f\\
v_{x} & v_{y}+g
\end{bmatrix}.
$
Assume $\tilde u(x,y)$ is the harmonic part of $u$ in $R,$ i.e., it is the unique solution of the Dirichlet boundary value problem
\begin{equation}
\label{4.5}
\begin{cases}
\Delta \tilde u(x,y)=0, & (x,y)\in R\\
\tilde u(x,y)=u(x,y), & (x,y)\in \partial R.
\end{cases}
\end{equation}
The Poincar\'e inequality gives the bound $\|u-\tilde u\|\leq h\|\nabla(u-\tilde u)\|.$ Multiplying the identity
$\Delta (u-\tilde u)=u_{xx}+u_{yy}=(e_{11}(\BM_{f,g})-e_{22}(\BM_{f,g}))_{x}+(2e_{12}(\BM_{f,g}))_{y}+g_x-f_y$
by $u-\tilde u$ we get by the Schwartz inequality the bounds
\begin{equation}
\label{4.6}
\|\nabla(u-\tilde u)\|\leq C\left[\|e(\BM_{f,g})\|+h(\|f_y\|+\|g_x\|)\right],\quad \|u-\tilde u\|\leq Ch\left[\|e(\BM_{f,g})\|+h(\|f_y\|+\|g_x\|)\right].
\end{equation}
In the next step we utilize the fact that $\tilde u$ is harmonic, thus we can apply Lemma~\ref{3.1} to $\tilde u.$ First apply the triangle inequality to get $\|u_{y}+f\|^2\leq 4(\|u_{y}-\tilde u_{y}\|^2+\|\tilde u_{y}\|^2+\|f\|^2),$ and then we apply Lemma~\ref{3.1} to the summand $\|\tilde u_{y}\|^2$ first and then the triangle inequality several times (also taking into account the bounds (\ref{4.6})) to get the
estimate
\begin{align}
\label{4.7}
\|u_{y}+f\|^2&\leq C\left(\frac{1}{h}\|u\|\cdot\|e(\BM_{f,g})\|+\|u\|(\|f_y\|+\|g_x\|)+\|u\|^2+\|e(\BM_{f,g})\|^2+\|f\|^2\right).
\end{align}
For the special case $f=\alpha\cdot\BU$ and $g=\beta\cdot\BU+w$ one has the bounds
$\|f_y\|\leq C\|U\|_{H^1(R)}\leq C(\|\BM_{f,g}\|+\|\BU\|+\|w\|),$ and
$\|g_x\|\leq C\|U\|_{H^1(R)}+\|w_x\|\leq C(\|\BM_{f,g}\|+\|\BU\|+\|w_x\|),$
thus an application of the Cauchy-Schwartz inequality (involving the parameter $\epsilon$) leads (\ref{4.7}) to (\ref{4.4}).
\end{proof}
\textbf{The block $13$.} For the block $13$ we freeze the variable $\Gth$ and deal with two-variable functions. We aim to prove that for any $\epsilon>0$ the estimate holds:
\begin{equation}
\label{4.8}
\|F_{13}\|^2+\|F_{31}\|^2\leq C\left(\frac{\|u_t\|\cdot\|e(\BF)\|}{h}+\|e(\BF)\|^2+\frac{1}{\epsilon}\|\Bu\|^2+\epsilon \|F_{21}\|^2\right),
\end{equation}
where the norms are over the whole shell $\Omega.$
\begin{proof}
Indeed, it is not difficult to see that (\ref{4.8}) follows from  Lemma~\ref{lem:4.1} with the following choice: Fix $\Gth\in (0,\omega)$ and consider the displacement $\BU=(u_t,A_zu_z),$ the vector fields $\alpha=(0,-A_z\Gk_z),$ $\beta=(A_z^2\Gk_z,-A_{z,z})$ and the function $w=\frac{A_zA_{z,\Gth}}{A_\Gth}u_\Gth$ in the variables $t$ and $z$ over the thin rectangle $R=(-h/2,h/2)\times(z^1(\Gth),z^2(\Gth)).$
\end{proof}
\textbf{The block $12$.} The role of the variables $\Gth$ and $z$ is the completely the same, thus we have an analogous estimate
\begin{equation}
\label{4.9}
\|F_{12}\|^2+\|F_{21}\|^2\leq C\left(\frac{\|u_t\|\cdot\|e(\BF)\|}{h}+\|e(\BF)\|^2+\frac{1}{\epsilon}\|\Bu\|^2+\epsilon \|F_{31}\|^2\right).
\end{equation}
Consequently adding (\ref{4.8}) and (\ref{4.9}) and choosing the parameter $\epsilon>0$ small enough we discover
\begin{equation}
\label{4.10}
\|F_{12}\|^2+\|F_{21}\|^2+\|F_{13}\|^2+\|F_{31}\|^2\leq C\left(\frac{\|u_t\|\cdot\|e(\BF)\|}{h}+\|e(\BF)\|^2+\|\Bu\|^2\right).
\end{equation}
A combination of (\ref{4.1}) and (\ref{4.10}) completes the proof of the lower bound.
It remains to note that one gets (\ref{2.1}) from that with $\BF$ in place of $\nabla\Bu$ by an application of the obvious bounds
$\|\BF-\nabla\Bu\|\leq h\|\BF\|$ and $\|e(\BF)-e(\Bu)\|\leq h\|\nabla\Bu\|.$ The Ansatz realising the asymptotics of $h$ in (\ref{2.3}) and (\ref{2.4}) has been constructed in [\ref{bib:Harutyunyan.2}] and reads as follows:
\begin{equation}
\label{4.11}
\begin{cases}
u_t=W(\frac{\Gth}{\sqrt{h}},z)\\
u_\Gth=-\frac{t\cdot W_{,\Gth}\left(\frac{\Gth}{\sqrt h},z\right)}{A_\Gth{\sqrt h}}\\
u_z=-\frac{t\cdot W_{,z}\left(\frac{\Gth}{\sqrt h},z\right)}{A_z},
\end{cases}
\end{equation}
where $W(z,y)\colon\mathbb R^2\to\mathbb R$ is a smooth and periodic in $x$ function that the derivative $W_x(x,y)$ is not identically zero. The calculation is omitted here.
\end{proof}

\end{document}